\documentclass[12pt,a4paper]{amsart}
\usepackage{geometry}
\usepackage{amsmath,amssymb,amsfonts,amscd,amsthm,wasysym,xcolor,enumitem,indentfirst,graphicx,booktabs,mathtools}
\usepackage[bookmarksnumbered,colorlinks,linktocpage,pagebackref,plainpages]{hyperref}
\hypersetup{pdfstartview={FitH}}

\numberwithin{equation}{section}
\newtheorem{theorem}{Theorem}[section]
\newtheorem{corollary}[theorem]{Corollary}
\newtheorem{lemma}[theorem]{Lemma}
\newtheorem{proposition}[theorem]{Proposition}

\theoremstyle{definition}

\newtheorem{definition}{Definition}

\allowdisplaybreaks[4]
\renewcommand{\arraystretch}{1.2}

\makeatletter
\@namedef{subjclassname@2020}{\textup{2020} Mathematics Subject Classification}
\makeatother

\begin{document}
\title[On the uniform flatness of polynomial graphs]{On the uniform flatness of polynomial graphs}

\author[Xieping Wang]{Xieping Wang$^1$}
\address{CAS Wu Wen-Tsun Key Laboratory of Mathematics and School of Mathematical Sciences, University of Science and Technology of China, Hefei 230026, Anhui, People's Republic of China}
\email{xpwang008@ustc.edu.cn}

\thanks{$^1$ Partially supported by the NSFC (Grant No. 12371083) and the Fundamental Research Funds for the Central Universities (Grant Nos. WK0010000099 and WK3470000030).}

\author[Li Zhang]{Li Zhang$^2$}
\address{School of Mathematics and Physics, Anhui Jianzhu University, Hefei 230601, Anhui, People's Republic of China}
\email{zhang12@mail.ustc.edu.cn}

\thanks{$^2$ Partially supported by the NSF of Anhui Province (Grant No. 2108085QA04).}


\begin{abstract}
This paper grew out of an effort to understand the so-called uniform flatness, an important concept in the Bargmann--Fock interpolation theory developed by Varolin et al. We show that the graph of every polynomial in one complex variable is a uniformly flat algebraic curve in $\mathbb C^2$, with vanishing upper density with respect to the Gaussian weight. Combined with a result of Pingali and Varolin, this implies that such graphs are interpolating for the Bargmann--Fock space on $\mathbb C^2$.
\end{abstract}
\maketitle

\section{Introduction}
A fundamental topic in complex analysis is the interpolation problem of holomorphic functions in weighted Bergman spaces on complex manifolds. To formulate this problem in a simple but typical case, we consider the Bargmann--Fock space on $\mathbb C^n$, that is the space of all entire functions $F$ on $\mathbb C^n$ satisfying
   $$
   \int_{\mathbb C^n} |F|^2\,e^{-|\,\cdot\,|^2}<\infty.
   $$
In this case the problem is to characterize those analytic hypersurfaces $W\subset \mathbb C^n$ with the interpolation property that for every holomorphic function $f$ on $W$ with $\int_W |f|^2\,e^{-|\,\cdot\,|^2}<\infty$, there is an entire function $F$ on $\mathbb C^n$ that extends $f$ and satisfies
$\int_{\mathbb C^n} |F|^2\,e^{-|\,\cdot\,|^2}<\infty$, by the geometric properties of $W$. In the case $n=1$ this problem has a very satisfactory solution, culminating in the celebrated Seip--Wallst\'{e}n theorem that a discrete subset $W\subset \mathbb C$ is interpolating if and only if it is uniformly separated and its upper density is less than an exact threshold. For this and other related results, as well as their applications in mathematics itself and beyond, we refer the reader to Seip's monograph \cite{Seip_book04} and the references therein.

In higher dimensions, Varolin and his collaborators have systematically investigated the above problem using the $L^2$ method for the $\bar{\partial}$-operator in a series of papers since 2006, including \cite{OCS-Varolin_MA06, PV_Crelle16} (see also \cite{Ohsawa_IV, Ohsawa_V} and \cite{Bern-Ortega_Crelle95} for previous related work in this direction). As it turns out, the interpolation problem becomes subtle in the higher-dimensional case and, despite considerable effort, is still far from well-understood. Nevertheless, Varolin et al. succeeded in introducing an important concept of {\it uniform flatness} for $W$ (to be recalled in Section \ref{sect: Uniform-flatness}), which together with a natural density condition on $W$ constitutes a sufficient condition for $W$ to have the interpolation property described above, and is even conjectured to be necessary when further assuming that $W$ is algebraic; see \cite{OCS-Varolin_MA06, PV_Crelle16, PV_Nagoya21, MPV_Nagoya23} for details.

The purpose of this paper is to prove the following

\begin{theorem}\label{thm:Uniform-flatness}
For every polynomial $P$ on $\mathbb C$ its graph
   $$
   \mathcal{C}=\big\{(z,\, \zeta)\in \mathbb C^2 \!: \zeta=P(z) \big\}
   $$
is a uniformly flat algebraic curve in $\mathbb C^2$, with vanishing upper density with respect to the Gaussian weight $|\cdot|^2$.
\end{theorem}

In sharp contrast, the graph of a complex polynomial in two or more variables is not necessarily uniformly flat and a simple counterexample to this is
   $$
   \mathcal{S}=\big\{(z,\, w, \, \zeta)\in \mathbb C^3 \!: \zeta=zw^2 \big\}.
   $$
It should also be mentioned that there are many smooth algebraic curves in $\mathbb C^2$ that are not uniformly flat; see \cite{PV_Nagoya21}.

Combining Theorem \ref{thm:Uniform-flatness} with \cite[Theorem 1]{PV_Crelle16} we immediately obtain the following

\begin{theorem}\label{thm:interpolating}
The graph of every polynomial on $\mathbb C$ is interpolating for the Bargmann--Fock space on $\mathbb C^2$.
\end{theorem}

The part of Theorem \ref{thm:Uniform-flatness} concerning the uniform flatness of $\mathcal{C}$ was claimed in \cite[Remark 1.3]{PV_Nagoya21} without proof. To gain a better understanding of the differences between algebraic and analytic hypersurfaces presented in the geometric interpolation theory of holomorphic functions, we provide a detailed  proof for this fact by quantifying the usual proof of the tubular neighborhood theorem in Riemannian geometry (which does not seem to be trivial, as will be seen in Section \ref{sect: Uniform-flatness}). Hopefully this will shed some new light on the conjecture mentioned above.

The other part of Theorem \ref{thm:Uniform-flatness} concerning the upper density of $\mathcal{C}$ is just a special case of the general fact that all algebraic hypersurfaces in $\mathbb C^n$ have vanishing upper density with respect to the Gaussian weight, regardless of their singularities, as we shall prove in Section \ref{sect: density}.

\smallskip
\noindent {\bf Acknowledgements.}
The first-named author would like to thank Prof. Vamsi P. Pingali for patiently explaining his interesting work with Varolin \cite{PV_Crelle16}. Both authors are grateful to the anonymous
referee  for his/her careful reading of the paper and many valuable comments.

\bigskip

\section{Proof of Theorem \ref{thm:Uniform-flatness}: the uniform flatness part}\label{sect: Uniform-flatness}
We begin by recalling the notion of uniform flatness for smooth analytic hypersurfaces in $\mathbb C^n$.

\begin{definition}
A smooth analytic hypersurface $W\subset \mathbb C^n$ is said to be {\it uniformly flat} if there exists a constant $\varepsilon>0$ such that
  $$
  N_{\varepsilon}(W) \coloneqq \big\{ z\in \mathbb C^n \!: \mbox{dist}(z,\, W)<\varepsilon \big\}
  $$
forms a tubular neighborhood of $W$ in $\mathbb C^n$.
\end{definition}

More explicitly, if we write $W$ in the form $W=T^{-1}(0)$, where $T$ is an entire function on $\mathbb C^n$ such that ${\rm d}T$ is nowhere vanishing on $W$\footnote{Such $T$ always exists, since the Picard group of $\mathbb C^n$ is trivial.}, then the uniform flatness of $W$ amounts to the existence of a constant $\varepsilon>0$ such that the map
  $$
  W\times \Delta(0,\, \varepsilon) \to N_{\varepsilon}(W), \quad \, (z,\, t)\mapsto z+t \frac{\overline{{\rm grad}\, T(z)}}{|{\rm grad}\, T(z)|}
  $$
is a diffeomorphism, where $\Delta(0,\, \varepsilon)\subset \mathbb C$ denotes the disc centered at the origin with radius $\varepsilon$ and ${\rm grad}\, T$ the complex gradient of $T$:
   $$
   {\rm grad}\, T=\bigg(\frac{\partial T}{\partial z_1},\ldots, \frac{\partial T}{\partial z_n}\bigg).
   $$
Note that when $n=1$, $W$ is exactly a discrete set in $\mathbb C$, and it is uniformly flat if and only if it satisfies the separation condition
   $$
   \inf\!\big\{|z-w|\!: \, z, w\in W \,\, \mbox{and}\,\, z\neq w\big\}>0.
   $$
In this case, $W$ is usually said to be {\it uniformly separated}.

\smallskip

Next we prove the following simple lemma, which will be used in the proof of Theorem $\ref{thm:Uniform-flatness}$.
\begin{lemma}\label{Lem: polynomial}
  Let $P$ be a polynomial of degree $d\geq 1$ on $\mathbb C$. Then
  \begin{enumerate}[leftmargin=2.0pc, parsep=4pt]
   \item [{\rm(i)}] there exists a constant $C>1$ such that
         $$
         C^{-1}|w|^{d-1} \leq \left|\frac{P(z)-P(w)}{z-w}\right| \leq C|w|^{d-1}
         $$
      holds for all  $|w|>>1$ and $0<|z-w|<1$;

   \item [{\rm(ii)}]
         $$
         \sup\limits_{0<|z-w|<1}\frac{|P'(z)-P'(w)|}{|z-w|+|P(z)-P(w)|}<\infty.
         $$
  \end{enumerate}
\end{lemma}

\begin{proof} We write $P(z)=\sum_{k=0}^{d}a_kz^k$ with $a_d\neq 0$, and first prove part {\rm(i)}. To this end, an application of the fundamental theorem of calculus and the triangle inequality yields
  $$
  \left|\sum_{k=0}^{d-1} z^k w^{d-1-k} \right| \geq d \left( \max\{|z|,\, |w|\} \right)^{d-1} \left(1-\frac{d-1}{2} \frac{|z-w|} {\max\{|z|,\, |w|\}} \right),
  \quad z,\, w\in\mathbb C\!\setminus\!\{0\}.
  $$
Hence
  $$
  \left|\sum_{k=0}^{d-1} z^k w^{d-1-k} \right| \geq \frac{d}{2}|w|^{d-1}   \qquad \mbox{whenever} \quad \; |w|\geq d  \; \; \; \mbox{and}  \; \; \; |z-w|\leq 1.
  $$
Applying this and the triangle inequality to the identity
  $$
  \frac{P(z)-P(w)}{z-w} = \sum_{k=1}^{d} a_k\bigg(\sum_{j=0}^{k-1} z^j w^{k-1-j}\bigg),
  $$
we then arrive at
   \begin{equation}\label{ineq: growth-est1}
    \begin{split}
     \left|\frac{P(z)-P(w)}{z-w}\right| &\geq |a_d| \left|\sum_{j=0}^{d-1} z^j w^{d-1-j} \right| - \sum_{k=1}^{d-1} |a_k| \sum_{j=0}^{k-1} |z|^j |w|^{k-1-j}\\
     &\geq |a_d| \left|\sum_{j=0}^{d-1} z^j w^{d-1-j} \right| - \sum_{k=1}^{d-1} |a_k| \big( (|w|+1)^k-|w|^k \big) \\
     &\geq \frac{d|a_d|}{4}|w|^{d-1}
    \end{split}
   \end{equation}
whenever $|w|>>1$ and  $|z-w|\leq 1$. On the other hand, it is easy to derive
   \begin{equation}\label{ineq: growth-est2}
    \begin{split}
     \left|\frac{P(z)-P(w)}{z-w}\right|
      &\leq \sum_{k=1}^{d} |a_k| \left|\frac{z^k-w^k}{z-w}\right| \leq \sum_{k=1}^{d} k|a_k| (|w|+1)^{k-1} \\
      &\leq \bigg(2^{d-1}\sum_{k=1}^{d} k|a_k|\bigg) |w|^{d-1}
    \end{split}
   \end{equation}
for all $|w|\geq 1$ and $|z-w|\leq 1$. This together with \eqref{ineq: growth-est1} proves part {\rm(i)}.

\smallskip

  As for part {\rm(ii)}, it suffices to find a uniform upper bound for
    $$
    \frac{|P'(z)-P'(w)|}{|z-w|+|P(z)-P(w)|}=\frac{|P'(z)-P'(w)|/|z-w|}{1+|P(z)-P(w)|/|z-w|}
    $$
when $|w|>>1$ and $0<|z-w|\leq 1$. But this is easy to achieve, in view of \eqref{ineq: growth-est1} and the estimate
   $$
   \left|\frac{P'(z)-P'(w)}{z-w}\right| \leq \bigg(\sum_{k=2}^{d} k(k-1)|a_k|\bigg) (|w|+1)^{d-2}
   \qquad \mbox{with}  \; \; \; |z-w|\leq 1,
   $$
which follows in a similar way to \eqref{ineq: growth-est2}. The proof is complete.
\end{proof}

We are now ready to prove Theorem $\ref{thm:Uniform-flatness}$.

\begin{proof}[Proof of Theorem $\ref{thm:Uniform-flatness}$]
Clearly, $\mathcal{C}$ is an algebraic curve in $\mathbb C^2$. As an immediate consequence of Proposition \ref{prop: density-vanishing} below, we also know that $\mathcal{C}$ has vanishing upper density with respect to $|\cdot|^2$.

Now it remains to show that $\mathcal{C}$ is uniformly flat. This will be done by quantifying the usual proof of the tubular neighborhood theorem in Riemannian geometry, for which we refer for example to \cite[pp. 133--135]{Lee_RGbook} (and the reader should take a closer look at it for a better understanding of the argument that follows).

 To begin with, consider the map $f\!:\mathbb C^2\to \mathbb C^2$ given by
   $$
   f(z,\, t)=\big(z,\, P(z)\big)+t\frac{\big(-\overline{P'(z)},\, 1\big)}{\displaystyle\sqrt{1+|P'(z)|^2}}.
   $$
As indicated by the argument on pp. 133--134 of Lee's book \cite{Lee_RGbook}, what we need to show here is that there exists a constant $\varepsilon>0$ such that $f$ is injective on
  \begin{equation*}\label{defn: tub-nbhd}
    V_{\varepsilon}(w)\coloneqq \Big\{z\in \mathbb C\!: |z-w|^2+|P(z)-P(w)|^2<\varepsilon^2\Big\}\times\Delta(0,\, \varepsilon)
  \end{equation*}
for all $w\in\mathbb C$. Once this is proved, it follows easily that $f$ is a diffeomorphism from $\mathbb C\times\Delta(0,\,\varepsilon/2)$ to its image, giving the desired result (compare with \cite[p. 135]{Lee_RGbook}).

In order to show the existence of a constant $\varepsilon>0$ with the above property, it is useful to observe first that Lemma \ref{Lem: polynomial} {\rm(i)} implies that there exists a constant $C>1$ such that
 \begin{equation}\label{ineq: ball-comparison}
    V_{\delta/ C^2}(w) \subseteq \Delta\bigg(w,\, \frac{\delta}{C(1+|w|^{d-1})}\bigg)\times\Delta(0,\, \delta) \subseteq V_{\delta}(w)
 \end{equation}
for all $w\in\mathbb C$ and $0<\delta<1$, where $d={\rm deg}\, P$ (here and henceforward we assume $d\geq1$, the desired result being trivial when $d=0$).
Next, a straightforward calculation shows that the (complexification of the real) Jacobian matrix of $f$ takes the following block form:
  \begin{equation}\label{eq:tub-nbhd_Jacobian}\renewcommand{\arraystretch}{1.5}
    Df(z,\, t)=\begin{pmatrix}
        A(z)+t B(z)         &      t C(z)  \\
        \overline{t C(z)}   &   \overline{A(z)+t B(z)} \\
   \end{pmatrix},
   \end{equation}
where
   \begin{equation*}\label{eq:tub-nbhd_Jacobian1}\renewcommand{\arraystretch}{2.2}
     A(z)=\begin{pmatrix}
         1      & \quad  \displaystyle\frac{-\overline{P'(z)}}{\sqrt{1+|P'(z)|^2}} \\
        P'(z)   & \quad  \displaystyle{\frac{1}{\sqrt{1+|P'(z)|^2}}} \\
   \end{pmatrix},
   \qquad \quad     \renewcommand{\arraystretch}{1.3}
     B(z)=\frac{\overline{P'(z)} P''(z)}{2(1+|P'(z)|^2)^{3/2}}
     \begin{pmatrix}
        \overline{P'(z)}  & \; 0 \\
              -1          & \; 0 \\
     \end{pmatrix}
   \end{equation*}
and
   \begin{equation*}\label{eq:tub-nbhd_Jacobian2}\renewcommand{\arraystretch}{1.3}
      C(z)=-\frac{\overline{P''(z)}}{2(1+|P'(z)|^2)^{3/2}}
      \begin{pmatrix}
        2+|P'(z)|^2  & \; 0 \\
          P'(z)      & \; 0 \\
      \end{pmatrix}.
   \end{equation*}
It follows in particular that $Df(z,\, 0)$ is always invertible and its inverse $Df(z,\, 0)^{-1}$ has operator norm one, as can be easily checked. Now for each fixed $w\in\mathbb C$, consider the family of maps
  $$
  \varphi_{\lambda}\!:\mathbb C^2\to \mathbb C^2, \quad  (z,\, t)\mapsto (z,\, t)+\big(\lambda-f(z,\, t)\big)Df(w,\, 0)^{-1}
  $$
parameterized by $\lambda\in\mathbb C^2$. Note that a point $(z,\, t)\in \mathbb C^2$ satisfies $f(z,\, t)=\lambda$ if and only if it is a fixed point of $\varphi_{\lambda}$. Using this and the fact that
  $$
  D\varphi_{\lambda} (z,\, t)={\rm Id}_{\mathbb C^4}-Df(z,\, t)Df(w,\, 0)^{-1},
  $$
and also taking \eqref{ineq: ball-comparison} into account, we conclude that all the  $\varphi_{\lambda}$'s are contractive on
$\Delta \big(w,\, \delta/C(1+|w|^{d-1})\big)\times\Delta(0,\, \delta)$ (and consequently $f$ is injective there, in particular on $V_{\varepsilon}(w)$ with $\varepsilon=\delta/ C^2$), provided $\delta>0$ is chosen so small that
  \begin{equation}\label{ineq: oper-norm_control}
    \sup_{(z,\, t)\in V_{\delta}(w)}\big\|Df(z,\, t)-Df(w,\, 0)\big\|\leq \frac12.\footnote{As usual, $\|M\|$ here denotes the operator norm of a matrix $M\in \mathbb C^{n\times n}$; and $\|M\|_2$ denotes its Euclidean norm, as we shall see next in \eqref{ineq: oper-norm_est1} below.}
  \end{equation}
In this way we reduce the problem to proving that there does exist a constant $\delta>0$ such that \eqref{ineq: oper-norm_control} holds for all $w\in\mathbb C$.

To seek such a constant $\delta$, we use \eqref{eq:tub-nbhd_Jacobian} to derive the following estimate:
  \begin{equation}\label{ineq: oper-norm_est1}
     \begin{split}
      \big\|Df(z,\, t)-Df(z,\, 0)\big\|&=|t|
      \left\|\renewcommand{\arraystretch}{1.25}\begin{pmatrix}
                B(z)          &        C(z)  \\
           \overline{C(z)}    &   \overline{B(z)} \\
        \end{pmatrix}\right\|
      \leq \sqrt2\, \Big(\|B(z)\|^2_2+\|C(z)\|^2_2\Big)^{\frac12} |t| \\
      & \leq \frac{2|P''(z)|}{(1+|P'(z)|^2)^{\frac12}}|t|\leq {\rm const.}\,|t|
     \end{split}
  \end{equation}
for all $z,\, t\in\mathbb C$, where the involved constant depends only on the polynomial $P$. Next we deal with the term
  $$
  \big\|Df(z,\, 0)-Df(w,\, 0)\big\|=\|A(z)-A(w)\|.
  $$
Since each entry of the matrix $A$ is a Lipschitz function of $P'$, it follows that
  $$
  \|A(z)-A(w)\| \leq \|A(z)-A(w)\|_2 \leq {\rm const.}\,|P'(z)-P'(w)|
  $$
for all $z,\, w\in\mathbb C$. In view of this and \eqref{ineq: oper-norm_est1}, as well as the triangle inequality, what now remains is to show
  $$
  \lim\limits_{\delta \to 0^+}\sup\Big\{|P'(z)-P'(w)|\!: z,\, w\in\mathbb C\; \;\mbox{and}\;\; |z-w|+|P(z)-P(w)|<\delta\Big\}=0.
  $$
But this follows immediately from Lemma \ref{Lem: polynomial} {\rm(ii)}.  The proof is complete.
\end{proof}

\medskip

\section{Density of affine algebraic hypersurfaces}\label{sect: density}
In this section we investigate the density property of algebraic hypersurfaces in $\mathbb C^n$ by using their own metric properties. To this end, let
  $$
  \omega=\frac{\sqrt{-1}}{2}\partial\overline{\partial}|z|^2
  $$
denote the standard K\"{a}hler form on $\mathbb C^n$. Given $z\in\mathbb C^n$ and $r>0$, we denote by $B(z, \,r)$ the Euclidean ball of center $z$ and radius $r$. Then we have the following

\begin{proposition}\label{prop: area-growth}
Let $A$ be a pure $p$-dimensional algebraic variety in $\mathbb C^n$, and let $[A]$ denote the current of integration over $A$. Then
\begin{equation}\label{ineq: area-growth}
  \frac{1}{r^{2p}} \int_{B(z, \,r)} [A]\wedge \omega^p \leq \pi^p {\rm deg}\,A
\end{equation}
for all $z\in\mathbb C^n$ and $r>0$, where ${\rm deg}\,A$ denotes the degree of $A$.
\end{proposition}

Two remarks are in order:
\begin{enumerate}[leftmargin=1.6pc, parsep=2pt]
  \item [{\rm(i)}]  The estimate in \eqref{ineq: area-growth} seems to be widely-known, probably going back to Stoll \cite{Stoll_algebraicity} (see also \cite[Section 17.2]{Chirka_CAbook}), who showed that the algebraicity of a pure $p$-dimensional analytic variety $A\subset\mathbb C^n$ can be characterized by the growth condition
        $$
        \sup_{r>0}\frac{1}{r^{2p}} \int_{B(z, \,r)} [A]\wedge \omega^p < \infty
        $$
      for some (and hence all) $z\in\mathbb C^n$.

  \item [{\rm(ii)}] The usual projection technique in the theory of complex analytic varieties (e.g., combining the corollary on p. 226 (and its proof) with Corollary 1 on p. 126 of Chirka's book \cite{Chirka_CAbook}) easily yields the following weaker estimate than \eqref{ineq: area-growth}:
       $$
       \frac{1}{r^{2p}} \int_{B(z, \,r)} [A]\wedge \omega^p \leq \binom{n}{p}\pi^p {\rm deg}\,A.
       $$
       This is actually already sufficient for the proof of Proposition \ref{prop: density-vanishing} below, but the estimate in \eqref{ineq: area-growth} is interesting in its own right because of its sharpness.

\end{enumerate}

\begin{proof}[Proof of Proposition $\ref{prop: area-growth}$]
We include a proof here for completeness. According to a well-known monotonicity result of Lelong, for every $z\in\mathbb C^n$ the function
  $$
  (0,\, \infty)\ni r \mapsto \frac{1}{r^{2p}} \int_{B(z, \,r)} [A]\wedge \omega^p
  $$
is increasing. Thus we need only show that its limit as $r\to\infty$ does not exceed $\pi^p {\rm deg}\,A$.

We show that the limit is in fact exactly $\pi^p {\rm deg}\,A$ for all $z\in\mathbb C^n$. Observe first that the fact that
  $$
  B(0, \,r-|z|)\subseteq B(z, \,r) \subseteq B(0, \,r+|z|)
  $$
reduces the problem to the case $z=0$. Now an application of Demailly's formula (see \cite[Chapter III, Formula 5.5]{Demaillybook}) yields
  $$
  \frac{1}{(\pi r^2)^p} \int_{B(0, \,r)} [A]\wedge \omega^p
  =\bigg(1+\frac{1}{r^2}\bigg)^p \int_{B(0, \,r)} [A]\wedge
   \bigg(\frac{\sqrt{-1}}{2\pi}\partial\overline{\partial}\log(1+|\cdot|^2)\bigg)^p,
  $$
and the right-hand side converges to the integral
  $$
  \int_{\mathbb C^n} [A]\wedge \bigg(\frac{\sqrt{-1}}{2\pi}\partial\overline{\partial}\log(1+|\cdot|^2)\bigg)^p
  $$
as $r\to\infty$. Note that this integral represents nothing but the projective volume of $A$ (or rather the volume of the closure of $A$ in $\mathbb{CP}^n$ --- a pure $p$-dimensional algebraic variety --- w.r.t. the Fubini--Study metric, hence finite), which coincides with the degree ${\rm deg}\,A$ of $A$ itself, in view of a standard result in complex algebraic geometry (cf. \cite[Section 14.7]{Chirka_CAbook}).
Consequently
  $$
  \lim_{r\to\infty} \frac{1}{r^{2p}} \int_{B(0, \,r)} [A]\wedge\omega^p = \pi^p {\rm deg}\,A,
  $$
and we are done.
\end{proof}

Let $W$ be a (possibly singular) analytic hypersurface in $\mathbb C^n$. Following \cite{OCS-Varolin_MA06, PV_Crelle16}, we associate with $W$ a family of $(1,\, 1)$-currents $\{\Upsilon^W_r\}_{r>0}$ on $\mathbb C^n$:
  $$
  \Upsilon^W_r \coloneqq [W]\ast \frac{\mathbf{1}_{B(0, \,r)}}{\mathrm{Vol}(B(0, \,r))},
  $$
where $\mathbf{1}_E$ denotes the characteristic function of a set $E\subseteq\mathbb C^n$ and $\ast$ is convolution. Note that the coefficients of each $\Upsilon^W_r$ are locally bounded functions on $\mathbb C^n$; in fact for each $z\in\mathbb C^n$, $\Upsilon^W_r(z)$ is the average of the current of integration $[W]$ over $B(z, \,r)$.

Similarly as above, given a plurisubharmonic function $\varphi$ on $\mathbb C^n$ we set
  $$
  \varphi_r \coloneqq \varphi \ast \frac{\mathbf{1}_{B(0, \,r)}}{\mathrm{Vol}(B(0, \,r))}
  $$
for $r>0$. With these preparations we recall the following

\begin{definition}[see \cite{OCS-Varolin_MA06, PV_Crelle16}]
Suppose $W$ is an analytic hypersurface in $\mathbb C^n$, and $\varphi$ is a plurisubharmonic function on $\mathbb C^n$. The {\it upper density} of $W$ with respect to $\varphi$ is defined to be
  $$
  D_{\varphi}(W) \coloneqq \inf\!\Big\{\alpha\in (0,\, \infty]\!:
    \frac{\sqrt{-1}}{2\pi}\partial\overline{\partial}\varphi_r
    -\frac{1}{\alpha} \Upsilon^W_r \geq 0 \quad \mbox{for all}\;\; r>>1\Big\}.
  $$
\end{definition}

We are now in a position to show that the upper density of algebraic hypersurfaces in $\mathbb C^n$ with respect to the Gaussian weight $|\cdot|^2$ is always zero. In fact, we can prove the following slightly more general result (which was essentially known to Pingali and Varolin; see \cite[Proposition 2.2]{PV_Nagoya21}).

\begin{proposition}\label{prop: density-vanishing}
Let $\varphi$ be a plurisubharmonic function on $\mathbb C^n$ satisfying
\begin{equation}\label{ineq: lower-positivity}
  \sqrt{-1}\partial\overline{\partial}\varphi \geq \varepsilon \omega
\end{equation}
for some constant $\varepsilon>0$. Then every algebraic hypersurface in $\mathbb C^n$ has vanishing upper density with respect to $\varphi$.
\end{proposition}

\begin{proof}
Let $W\subset\mathbb C^n$ be an algebraic hypersurface. Observe that after identifying currents of top degree with distributions, the trace measure of the current $\Upsilon^W_r$ with respect to $\omega$ is
\begin{equation*}\label{eq:trace-formula}
   {\rm tr}_{\omega}\Upsilon^W_r(z)=\frac{n}{\pi^n r^{2n}} \int_{B(z, \,r)} [W]\wedge \omega^{n-1},
   \quad z\in\mathbb C^n.
\end{equation*}
Since  $\Upsilon^W_r$ is positive, Proposition \ref{prop: area-growth} implies
  $$
  \Upsilon^W_r \leq \big({\rm tr}_{\omega}\Upsilon^W_r\big)\omega \leq \frac{{\rm const}_{n,\, W}}{r^2}\,\omega.
  $$
Together with \eqref{ineq: lower-positivity}, this estimate in turn implies that for each $\alpha>0$,
  $$
  \frac{\sqrt{-1}}{2\pi}\partial\overline{\partial}\varphi_r - \frac{1}{\alpha} \Upsilon^W_r
  \geq \bigg(\frac{\varepsilon}{2\pi}-\frac{{\rm const}_{n,\, W}}{\alpha}\frac{1}{r^2}\bigg)\omega>0
  $$
provided $r>{\rm const}_{n, \,W, \,\alpha, \,\varepsilon}$. Hence $D_{\varphi}(W)=0$, as desired.
\end{proof}

The preceding proof does not use the full force of the assumption that $W\subset\mathbb C^n$ is an algebraic hypersurface; we need only assume the following condition:
  $$
  \lim\limits_{r \to \infty}\big\|{\rm tr}_{\omega}\Upsilon^W_r\big\|_{L^{\infty}(\mathbb C^n)}=0.
  $$
A large number of transcendental examples of analytic hypersurfaces in $\mathbb C^n$ with this property can be easily constructed using the following proposition.

\begin{proposition}\label{prop:density-vanishing(trans-case)}
Let $\Gamma_1,\ldots, \Gamma_n$ be discrete sets in $\mathbb C$, and set
  $$
  W\coloneqq \bigcup_{k=1}^n \mathbb C^{k-1}\times\Gamma_k\times\mathbb C^{n-k}.
  $$
Then
  $$
  D_{|\,\cdot\,|^2}(W)\leq n \max_{1\leq k\leq n}D_{|\,\cdot\,|^2}(\Gamma_k).
  $$
\end{proposition}

\begin{proof}
First note that  the $(1,\, 1)$-currents $\{\Upsilon^{W_k}_r\}_{r>0}$ associated to each $W_k\coloneqq \mathbb C^{k-1}\times\Gamma_k\times\mathbb C^{n-k}$ are equal to
  \begin{equation}\label{eq: transcendental-density}
    \Upsilon^{W_k}_r(z)=\frac{\big(\mathrm{Vol}_{k-1}\times\delta_{\Gamma_k}\times\mathrm{Vol}_{n-k}\big)(B(z, \,r))}{\mathrm{Vol}_n(B(0, \,r))}
                      \frac{\sqrt{-1}}{2} {\rm d}z_k\wedge{\rm d}\bar{z}_k, \quad  r>0,
  \end{equation}
where $\delta_{\Gamma_k}\coloneqq\sum_{\zeta\in\Gamma_k}\delta_{\zeta}$ denotes the counting measure on $\Gamma_k$  and $\mathrm{Vol}_l$ the Lebesgue measure on $\mathbb C^l$ for each integer $l\geq 1$.

Writing \eqref{eq: transcendental-density} in a more explicit way we get\footnote{For a finite set $E$, $\sharp E$ denotes the number of its elements.}
  \begin{equation*}
  \begin{split}
    \Upsilon^{W_k}_r(z)
      &=\frac{n!}{(\pi r^2)^n}\sum_{\zeta\in\Gamma_k\cap\Delta(z_k,\, r)}\frac{\pi^{n-1}}{(n-1)!}(r^2-|\zeta-z_k|^2)^{n-1}
                         \frac{\sqrt{-1}}{2} {\rm d}z_k\wedge{\rm d}\bar{z}_k \\
      &\leq\frac{n}{\pi}\frac{\sharp\big(\Gamma_k\cap\Delta(z_k,\, r)\big)}{r^2}\frac{\sqrt{-1}}{2} {\rm d}z_k\wedge{\rm d}\bar{z}_k
  \end{split}
  \end{equation*}
for all $1\leq k\leq n$ and $r>0$, hence
  $$
  \Upsilon^{W}_r(z)=\sum_{k=1}^{n}\Upsilon^{W_k}_r(z)
  \leq \frac{n}{\pi} \sum_{k=1}^{n}\frac{\sharp\big(\Gamma_k\cap\Delta(z_k,\, r)\big)}{r^2}\frac{\sqrt{-1}}{2} {\rm d}z_k\wedge{\rm d}\bar{z}_k.
  $$
This in turn implies
  $$
  \frac{\sqrt{-1}}{2\pi}\partial\overline{\partial}|z|^2-\frac{1}{\alpha} \Upsilon^W_r(z)
  \geq \frac{1}{\pi} \sum_{k=1}^{n}\bigg(1-\frac{n}{\alpha}\frac{\sharp\big(\Gamma_k\cap\Delta(z_k,\, r)\big)}{r^2}\bigg)\frac{\sqrt{-1}}{2} {\rm d}z_k\wedge{\rm d}\bar{z}_k
  $$
for all $\alpha\in (0,\, \infty]$. Since $|\cdot|^2_r=|\cdot|^2+\frac{n}{n+1}r^2$ for $r>0$, the desired result now follows immediately from the definition of $D_{|\,\cdot\,|^2}(W)$ and the following formula:
  $$
  D_{|\,\cdot\,|^2}(\Gamma_k)=\limsup\limits_{r\to \infty}\sup_{\zeta\in\mathbb C}\frac{\sharp\big(\Gamma_k\cap\Delta(\zeta,\, r)\big)}{r^2}, \qquad k=1,\ldots,n.
  $$
\end{proof}

As an additional consequence of the preceding proposition and \cite[Theorem 1]{PV_Crelle16} we obtain the following
\begin{corollary}
Let $\Gamma_1,\ldots, \Gamma_n$ be uniformly separated discrete sets in $\mathbb C$ satisfying
  $$
  D_{|\,\cdot\,|^2}(\Gamma_k)<\frac 1n, \quad k=1,\ldots, n.
  $$
Then the analytic hypersurface
  $$
  W\coloneqq \bigcup_{k=1}^n \mathbb C^{k-1}\times\Gamma_k\times\mathbb C^{n-k}
  $$
is interpolating for the Bargmann--Fock space on $\mathbb C^n$.
\end{corollary}

\begin{proof}
Since each $\Gamma_k$ is uniformly separated, it is easy to see that $W$ is uniformly flat in the sense of Pingali--Varolin (see \cite[Definition 2.4]{PV_Crelle16}). Moreover, $D_{|\,\cdot\,|^2}(W)<1$ by Proposition \ref{prop:density-vanishing(trans-case)}. Then \cite[Theorem 1]{PV_Crelle16} applies.
\end{proof}

\end{document}